\documentclass[11pt,a4paper]{amsart}
\usepackage{amsfonts}
\usepackage{amsthm}
\usepackage{amsmath}
\usepackage{amscd}
\usepackage[latin2]{inputenc}
\usepackage{t1enc}
\usepackage[mathscr]{eucal}
\usepackage{indentfirst}
\usepackage{graphicx}
\usepackage{graphics}
\usepackage{pict2e}
\usepackage{epic}
\usepackage{amssymb}
\usepackage[colorlinks,citecolor=red,urlcolor=blue,bookmarks=false,hypertexnames=true]{hyperref} 
\numberwithin{equation}{section}
\usepackage[margin=3cm]{geometry}

\makeatletter
\@namedef{subjclassname@2020}{%
  \textup{2020} Mathematics Subject Classification}
\makeatother
\usepackage[T1]{fontenc}
\newtheorem{theorem}{Theorem}[section]

\newtheorem{lemma}[theorem]{Lemma}
\newtheorem{proposition}[theorem]{Proposition}

\theoremstyle{definition}

\newtheorem{remark}[theorem]{Remark}

\newtheorem*{xrem}{Theorem}
\numberwithin{equation}{section}

\begin{document}


\baselineskip=17pt


\title{on the Rankin-Selberg problem in families}

\author{Jiseong Kim}
\address{University at Buffalo, Department of Mathematics
244 Mathematics Building
Buffalo, NY 14260-2900}
\email{Jiseongk@buffalo.edu}

\date{}

\begin{abstract} 
In this paper, we investigate the Rankin-Selberg problem over short intervals in families of holomorphic modular forms and Hecke-Maass cusp forms. Our investigation assumes a Lindel\"of-on-average bound for holomorphic modular forms, and for Hecke-Maass cusp forms, we make no assumptions.
\end{abstract}
\maketitle

\section{Introduction}\

Let $\mathbb{H}=\{z=x+iy | x \in \mathbb{R}, y \in (0,\infty)\},$ and let  $G=SL(2,\mathbb{Z}).$ 
Define $j_{\gamma}(z)=(cz+d)^{-1},$ where 
$\gamma= \left(\begin{matrix}
a & b \\
c & d 
\end{matrix} \right) \in G.$
A holomorphic function $f: \mathbb{H} \rightarrow \mathbb{C}$ satisfying the equation
\begin{equation}f(\gamma z)= j_{\gamma}(z)^{-k}f(z) \;\;\; \textrm{for all }\;\; \gamma \in SL(2,\mathbb{Z}),\nonumber\end{equation}
is called a modular form of weight $k.$ It is well known that any such modular form $f(z)$ has a Fourier expansion at the cusp $\infty$ given by 
\begin{equation}
    f(z)=\sum_{n=0}^{\infty} b_{f}(n)e(nz)
\end{equation}
where $e(z)=e^{2\pi iz},$ and the normalized Fourier coefficient $a_{f}(n)$ of $f(z)$ is defined as
\begin{equation}
a_{f}(n):=b_{f}(n)n^{-(k-1)/2}.
\end{equation}
The set of all modular forms of a fixed weight $k$ forms a vector space, and we denote it by $M_{k}.$  
Furthermore, we focus on a subspace of this space, denoted by $C_{k}$, which comprises of all modular forms of a weight $k$ that have a zero constant term. Let us consider the $n$th Hecke operator, denoted by $T_{n}$, which is defined as 

$$ (T_{n}f)(z)= \frac{1}{\sqrt{n}} \sum_{ad=n} \sum_{b(\rm{mod} \thinspace \it{d})} f(\frac{az+b}{d})$$
for all $f \in C_{k}.$

It is well established that there exists an orthonormal basis of eigenfunctions, called Hecke cusp forms, for all Hecke operators $T_{n}$ in the space of modular forms $C_{k}$. In this paper, we will use $S_{k}$ to denote this orthonormal basis of $C_{k}$.

When $f$ is a Hecke cusp form, eigenvalues $\lambda_{f}(n)$ of the $n$th Hecke operator are real numbers that satisfy the following properties:
$$a_{f}(n)=a_{f}(1)\lambda_{f}(n),\;\;\textrm{ and } \;\;\lambda_{f}(m)\lambda_{f}(n)= \sum_{d|(m,n)} \lambda_{f}\left(\frac{mn}{d^{2}}\right).$$ 

For details, see \cite[Chapter 14]{IK1}. 
A function $u:G \setminus \mathbb{H} \rightarrow \mathbb{C}$ is called a Maass form of type $v \in \mathbb{C}$ if $u$ satisfies the following properties: 
\begin{itemize}
\item $ \displaystyle \iint_{G \textbackslash \mathbb{H}} |u(z)|^{2} \frac{dx dy}{y^{2}} <\infty,$
\item $u(\gamma z)= u(z)$ for all $\gamma \in G,$
\item $\Delta u=v(1-v)u,$ where $\displaystyle \Delta =-y^{2}\left( \frac{\partial^{2}}{\partial x^{2}} + \frac{\partial^{2}}{\partial y^{2}} \right),$
\item $ \displaystyle \int_{0}^{1} u(z) dx=0.$
\end{itemize}
When $u$ is a non-constant Maass form of type $v,$ we can express it using the Whittaker expansion given by
$$ u(z)= \sqrt{2\pi y} \sum_{n \neq 0} a_{u}(n)K_{v-\frac{1}{2}}(2\pi |n| y) e(nx).$$ For further details, see \cite[Chapter 3]{DG}.
We are interested in a particular type of Maass forms called Hecke-Maass forms, which are eigenfunctions of the Hecke operators $T_n$ for all $n$.

There have been several results on the asymptotic evaluations of sums related to Hecke eigenvalues. One of the most notable results is the Rankin-Selberg formula that was first established by Rankin \cite{RA1} and Selberg \cite{S}. The formula states that for a Hecke cusp form $f$,
\begin{equation}\label{Rankin-Selberg1}
    \sum_{1 \leq m \leq X}\lambda_{f}(m)^{2}=\frac{ \textrm{Res}_{s=1} L(\textrm{sym}^{2}f,s)}{\zeta(2)}X+O_{f}(X^{3/5})
\end{equation} 
where $L\left(\textrm{sym}^{2}f,s\right)$ is the symmetric square $L$-function of $f,$ $\zeta(s)$ is the Riemann zeta function (for the details, see \cite[5.12]{IK1}). The Rankin-Selberg
problem is to improve the exponent 3/5.
Recently, Huang \cite{huang2021rankinselberg} improved the exponent $3/5$ to $3/5-1/560+o(1)$. It is worth noting that these results apply not only to the holomorphic  $SL(2,\mathbb{Z})$ Hecke cusp forms but also to the  $SL(2,\mathbb{Z})$ Hecke-Maass cusp forms.
For convenience, we denote the coefficient of $X$ by $c_{1,f}.$ 
Ivic \cite{Ivicrankin} proved that for $1 \leq H \leq X,$ the following holds: 
$$
 \int_X^{2 X}\left( \sum_{x < m \leq x+H}\lambda_f(m)^2-c_{1,f}H \right)^2 d x \ll_{f,k} X^{\left(9+12\alpha\right)/\left(7+4\alpha\right)}H^{8/\left(7+4\alpha\right)}
$$
where $\alpha=\limsup_{t\rightarrow \infty} \frac{\log \left|\zeta(1/2+it)\right|}{\log t}.$ Note that under the Lindel\"{o}f hypothesis, we have $\alpha=0.$
In \cite{Ivicrankin}, Ivic also proved that, assuming the generalized Lindel\"{o}f hypothesis $\displaystyle L(\textrm{sym}^{2}f,\frac{1}{2}+it) \ll_{f,\epsilon} k^{\epsilon}(|t|+1)^{\epsilon}$, the following holds:
$$
 \int_X^{2 X}\left( \sum_{x < m \leq x+H}\lambda_f(m)^2-c_{1,f}H \right)^2 d x \ll_{f,k,\epsilon} X^{1+\epsilon}H^{4/3} \;\;\;\;\; (X^{\epsilon} < H < X^{1-\epsilon}).
$$
Recently, Mangerel \cite{Mangerel} proved the following theorem:
\begin{xrem} Let $10 \leq H_0 \leq X /(10 \log X)$ and set $H:=H_0 \log X$. Then there exist a constant $\theta>0$ such that
$$
\frac{1}{X} \int_X^{2 X}\left( \sum_{x<m \leq x+H}\lambda_f(m)^2-c_{1,f} H\right)^2 d x \ll H^{2}\left( \frac{\log \log H_0}{\log H_0}+\frac{\log \log X}{(\log X)^\theta}\right) .
$$
\end{xrem}
\noindent In order to prove Theorem \ref{theorem}, we apply the Petersson trace formula (see Lemma \ref{trace}) and assume the following Lindel\"of-on-average bound for holomorphic cusp forms:
\begin{equation}\label{lindelof}
\sum_{f \in S_{k}} \varphi_{f} | L(\textrm{sym}^{2}f,1/2+it)|^{2}\ll_{\epsilon}  k^{\epsilon}(|t|+1)^{\epsilon} \;\;\textrm{for} \;\;t \in \mathbb{R}, \; \textrm{for all} \; \epsilon \in (0,1), \end{equation}
where $$\varphi_{f}:=\left(\|f\|^{2}\sum_{g \in S_{k}} \frac{1}{\|g\|^{2}}\right)^{-1}.$$
\begin{theorem}\label{theorem} Assume the Lindel\"of-on-average bound \eqref{lindelof}.
 Let $\epsilon$ be a sufficiently small number in (0,1). Let $1\ll \log H < \log X < (\log k)/4.$ Then
\begin{equation}\label{thm}\begin{split}
\sum_{f \in S_{k}} \varphi_{f} \frac{1}{X} \int_X^{2 X}\left(\sum_{x<m \leq x+H} \lambda_{f}^2(m)-c_{1,f}H\right)^2 d x \ll_{\epsilon}  k^{\epsilon}H^{1+\epsilon}.
\end{split}\end{equation}
\end{theorem}
\noindent Note that $$\sum_{f \in S_{k}} \frac{1}{\|f\|^{2}} \sim \frac{(4\pi)^{k-1}}{\Gamma(k-1)} \;\textrm{as} \; k \rightarrow \infty, \;\; \frac{1}{\|f\|^{2}}\ll_{\epsilon} k^{-1+\epsilon}. $$

Let $u_{j}$ be an $SL(2,\mathbb{Z})$ Hecke-Maass cusp form with Laplacian eigenvalue  $\frac{1}{4}+t_{j}^{2},$ and let $\lambda_{u_{j}}(n)$ be the $n$th normalized Fourier coefficient of $u_{j}.$ The function $\lambda_{u_{j}}(n)$ also satisfies 
\begin{equation}\label{heckemaass}\lambda_{u_{j}}(m)\lambda_{u_{j}}(n)=\sum_{d|(m,n)} \lambda_{u_{j}}\left(\frac{mn}{d^{2}}\right). \end{equation}
As mentioned before, we have 
\begin{equation}\label{Rankin-Selberg}
    \sum_{1 \leq m \leq X}\lambda_{u_{j}}(m)^{2}=\frac{ \textrm{Res}_{s=1} L(\textrm{sym}^{2}u_{j},s)}{\zeta(2)}X+O_{u_{j}}(X^{3/5-1/560+o(1)}).
\end{equation}
For convenience, we denote the coefficient of $X$ by $c_{1,u_{j}}.$ 
 Our methods for Theorem 1.1 apply equally well to $SL(2,\mathbb{Z})$ Hecke-Maass cusp forms $u_{j}.$ For the Hecke-Maass cusp forms $u_{j},$ one needs to apply the Kuznetsov trace formula (see Lemma \ref{Kuznetsov}) instead of the Petersson trace formula. According to the result of Khan and Young \cite[Theorem 1.1, Theorem 1.3]{Pyoung}, for any $T>1$ and $0< \epsilon <2,$  we have 
\begin{equation}\label{khanyoung}
    \sum_{T \leq t_{j} \leq 2T} \left|L(\textrm{sym}^{2}  u_{j},1/2+iU)\right|^{2} \ll_{\epsilon}  T^{2+\epsilon} \;\;\textrm{for}\;\; U \in [0,(2-\epsilon)T].
\end{equation} Therefore, we unconditionally prove the following theorem.
\begin{theorem}\label{Theorem1.2} Let $\epsilon$ be a sufficiently small number in (0,1).
Let $1 \ll \log H < \log X,$ and let  $ 2\log X -(\log H)/2 < \log T.$ Then 
\begin{equation}\label{intge2}\sum_{T \leq t_{j} \leq 2T} \varpi_{j}
\frac{1}{X} \int_X^{2 X}\left(\sum_{x<m \leq x+H} \lambda_{u_{j}}^2(m)-c_{1,u_{j}}H\right)^2 d x\ll_{\epsilon} T^{\epsilon}H^{1+\epsilon}\end{equation}
where 
\begin{equation}w(j)=\frac{\zeta(2)}{L(\textrm{sym}^{2} u_{j},1)}, \;\;\; \varpi_{j}= w(j) \left(\sum_{T \leq t_{j} \leq 2T} w(j)\right)^{-1}.\nonumber \end{equation}\end{theorem}
\noindent Note that 
$$T^{2-\epsilon} \ll_{\epsilon} \sum_{T \leq t_{j} \leq 2T} w(j) \ll T^{2},\;\; w(j)\ll_{\epsilon} |t_{j}+1|^{\epsilon}$$ 
(see \cite{Hoffstein}). 
\begin{remark}
    As in \cite{GKMR}, one could consider problems related to the variance over arithmetic progressions. However, since we are considering families, these problems are relatively straightforward to handle without resorting to the methods described in \cite{GKMR}. For readers who are interested, we provide a simple example in the Appendix to illustrate this point.
\end{remark}
\subsection{Sketch of the proof}
\noindent Because the proofs of Theorem 1.1 and Theorem 1.2 share many similarities, we will provide a sketch of the proof of Theorem 1.1.
In \cite{GKMR}, Gorodetsky, Mat\"omaki, Radziwi\l\l \thinspace  and
Rodgers used the identity $\mu^2(m)=\sum_{d^2 \mid m} \mu(d)$ to evaluate the variance of the number of square-free integers in short intervals.
The squares of Hecke eigenvalues also exhibit a similar convolution structure, as shown by the following equation:
\begin{equation}\label{Hecke}\lambda_{f}\left(m\right)^{2}=\sum_{d \mid m} \lambda_{f}\left(d^{2}\right).\end{equation}
It should be noted that while the identity for square-free integers involves $d^2$, the convolution structure \eqref{Hecke} considers $d$. As a result, when dealing with square-free integers, we are essentially working with a sparse set of squares. However, when working with Hecke eigenvalues, we are considering a set of natural numbers that is not sparse. This phenomenon is explained well in the paper \cite{GMR} by Gorodetsky, Mangerel, and Rogers. 
 The authors of \cite{GKMR} split the sum over $d^{2}$ into $d^{2} \leq z$ and $d^{2}>z$ for some $z.$ We split the sum over $d$ into  $d \leq H$ and $d> H.$ 
By \eqref{Hecke}, we have 
$$\sum_{x<n\leq x+H} \lambda_{f}(n)^{2} 
= \sum_{1 \leq d \leq x+H}\lambda_{f}\left(d^{2}\right)\sum_{\frac{x}{d} < n \leq \frac{x+H}{d}}1. $$
When $d$ is small, it is expected that $$\sum_{x < n \leq x+H} \lambda_{f}(n)^{2} \sim H\sum_{d \leq x+H}\frac{\lambda_{f}\left(d^{2}\right)}{d}.$$ This can be demonstrated by using the Fourier expansion of
$$x \mapsto \sum_{\frac{x}{d} < n \leq \frac{x+H}{d}}1$$
(see Proposition \ref{Proposition1}). 
However, when $d$ is large (i.e. $d>H$), we have 
$$ \sum_{\frac{x}{d} < n \leq \frac{x+H}{d}} 1 = \delta_{[\frac{x}{d}]+1=[\frac{x+H}{d}]}$$ where
\begin{equation}
\delta_{{[\frac{x}{d}]+1=[\frac{x+H}{d}]}}:= \left. \begin{cases} 1 & \; \textrm{when} \;\; [\frac{x}{d}]+1=[\frac{x+H}{d}] \;\;  \\ 0 & \; \textrm{otherwise}  \end{cases} \right\}.
 \nonumber\end{equation}
As a result, we encounter a double sum
\begin{equation}
 \sum_{H< d_{1} \leq x+H}\sum_{H< d_{2} \leq x+H} \lambda_{f}\left(d_{1}^{2}\right)\lambda_{f}\left(d_{2}^{2}\right)
 \delta_{[\frac{x}{d_{1}}]+1=[\frac{x+H}{d_{1}}]} \delta_{[\frac{x}{d_{2}}]+1=[\frac{x+H}{d_{2}}]},
\nonumber\end{equation} which is difficult to control. To avoid this, we take a different approach by considering the average over all modular forms $f$ of weight $k$ and applying the Petersson trace formula (see Lemma \ref{trace}).
For the remaining parts, we utilize the Lindel\"of-on-average bound to obtain an approximation of the form
$$c_{1,f} \sim \sum_{m \leq  H^{2}} \frac{\lambda_{f}\left(m^{2}\right)}{m}$$ with a small error term (see Proposition \ref{RankinProposition35}).
\begin{remark}
Our approach clearly applies to demonstrating analogous results for a linear space $S_k(N)$ of cusp forms with weight $k$ and level $N.$
 Some of the results in our paper still hold if we replace $\lambda_{f}$ and $S_{k}$ with the coefficients of the standard $L$-function attached to a $GL(n)$ automorphic representation and a suitable family of $GL(n)$ automorphic representations, respectively.
 For instance, suppose we let $A_{g}(n,1,1,\dots,1)$ denote the $n$th coefficient of an $L$-function $L(g,s)$ associated with a $SL(n,\mathbb{Z})$ Hecke-Maass cusp form $g$. We can address the corresponding problem for $g$ by considering the expression:
$$\sum_{x < nd < x+H} A_{g}(d,1,\dots,1)- L(g,1)H.$$
Thus, using orthogonality (for instance, see \cite{Goldstadwood}), we can easily modify our method to establish that the averages of the following expressions over suitable families are bounded by the same upper bound as that of for the $SL(2,\mathbb{Z})$ Hecke-Maass (or holomorphic) forms:
\begin{equation}
\begin{split}
&\int_X^{2 X}\left(\sum_{\substack{x<nd\leq x+H \\ d\leq z}} A_{g}\left(d,1,\dots,1\right)-H \sum_{d\leq z} \frac{A_{g}\left(d,1,\dots,1\right)}{d}\right)^2 d x,\\
&\int_X^{2 X}\left(\sum_{\substack{x<nd\leq x+H \\ z<d}} A_{g}\left(d,1,\dots,1\right)\right)^2 d x
\end{split}
\end{equation}
Assuming the Lindel\"of-on-average bound for $L(g,s)$, we can  easily modify our method to show that
$$H L(g,1)-H \sum_{d \leq z} \frac{A_{g}(d,1,\dots,1)}{d}$$ is sufficiently small. 
\end{remark}
\begin{remark} It is conjectured (see \cite{Ivic22}) that 
$$ \sum_{1 \leq m \leq X} \lambda_{f}^{2}(m)-c_{1,f}X= O_{f,k}\left(X^{\frac{3}{8}+o(1)}\right).$$
 By Theorem \ref{theorem}, when $H=X^{3/8+\epsilon}$, we have that for almost all $f \in S_{k}$ and $x \in [X,2X],$ $$\left(\sum_{x<m \leq x+H} \lambda_f^2(m)-c_{1, f} H\right)^2 \ll_{\epsilon} k^{\epsilon}H^{1+\epsilon}\;\;\; \textrm{for sufficiently large} \;k.$$
 This leads to the conclusion that
$$\sum_{x<m \leq x+X^{3/8+\epsilon}} \lambda_f^2(m)-c_{1, f} X^{3/8+\epsilon} \ll_{\epsilon} k^{\epsilon}X^{3/16}$$
for almost all $f \in S_{k}$ and $x \in [X,2X].$ 
It is worth noting that the Lindel\"of hypothesis implies that when $H=X^{1/2+\epsilon}, $
$$\sum_{x<m \leq x+X^{1/2+\epsilon}} \lambda_f^2(m)-c_{1, f} X^{1/2+\epsilon} =o(X^{1/2+\epsilon}).$$   
\end{remark}

\section{Lemma}
\noindent
For convenience, we need to smooth the integrals in equations \eqref{thm} and \eqref{intge2}. We use the Beurling-Selberg polynomial from \cite{GKMR}. To simplify notation, we denote the number of the divisors of $m$ as $\tau(m)$. Additionally, for any condition $\mathcal{A}$, we define 
\begin{equation}
\delta_{\mathcal{A}}:= \left. \begin{cases} 1 & \; \textrm{when} \;\; \mathcal{A} \;\; \textrm{is satisfied} \\ 0 & \; \textrm{otherwise}  \end{cases} \right\}.
 \nonumber\end{equation} 
 We use the common notation that the least common multiple of $a$ and $b$, and the greatest common divisor of $a$ and $b$ are denoted as $[a,b]$ and $(a,b),$ respectively. Finally, we use the convention that $\epsilon$ denotes an arbitrarily small positive quantity that may vary from line to line.
\begin{lemma}\label{smoothing}
There exist functions $\sigma_{-}$ and $\sigma_{+}$ that map from $\mathbb{R}$ to $\mathbb{R}$ and satisfy the following properties: 
\begin{itemize}
\item The Fourier transforms $\hat{\sigma}_{-}(x):=\int_{\mathbb{R}}\sigma_{-}(y)e(-xy)dy$ and  $\hat{\sigma}_{+}(x):=\int_{\mathbb{R}}\sigma_{+}(y)e(-xy)dy$ have support 
$[-BH^{\frac{\epsilon}{2}},BH^{\frac{\epsilon}{2}}]$ for a sufficiently large absolute constant $B$.
\item $\sigma_{-} \leq  1_{[1,2]} \leq \sigma_{+},$ where 
$1_{[1.2]}$ is the indicator function over the interval $[1,2].$ 
\item   $\displaystyle   1- H^{-\frac{\epsilon}{2}} \leq  \int_{\mathbb{R}} \sigma_{-}(x) dx,  \int_{\mathbb{R}} \sigma_{+}(x) dx \leq 1+ H^{-\frac{\epsilon}{2}}.$
\end{itemize}
\end{lemma} 
\begin{proof}
See \cite[Chapter 1, page 6]{Montgomery2}.
\end{proof}
\noindent 
Since $\lambda_{f}(n)$ is divisor-bounded, we can apply the following lemma.
\begin{lemma}\label{shiu}
Let $\delta\in(0,1)$ be a fixed constant and suppose $X\geq Y\geq X^\delta\geq 2$. Then, for any positive integer $l$, we have the inequality
\begin{equation}\label{average2}\frac{1}{Y} \sum_{X<n \leq X+Y} \tau\left(n\right)^{l} \ll_{\delta} \left(\log X\right)^{2^{l}-1}.\end{equation}
\end{lemma}
\begin{proof}
Using Shiu's Theorem (\cite[Lemma 2.3,(i)]{MRT2}), we obtain the bound
$$ \frac{1}{Y} \sum_{X<n \leq X+Y} \tau\left(n\right)^{l} \ll_{\delta}  \prod_{p \leq X}\left(1+\frac{\tau\left(p\right)^{l}-1}{p}\right).$$
Since $\tau(p)=2$ for any prime $p$, we have
\begin{equation}\begin{split}\prod_{p \leq X}\left(1+\frac{\tau\left(p\right)^{l}-1}{p}\right) &=\prod_{p \leq X}\left(1+\frac{2^{l}-1}{p}\right) \\&\leq exp\left(\sum_{p \leq X} \frac{2^{l}-1}{p}\right).\end{split}\end{equation}
By Merten's theorem, we see that 
$$\sum_{p\leq X} \frac{1}{p} = \log\log X + C +o(1)$$
for some constant $C.$
Therefore,
$$ exp\left(\sum_{p \leq X} \frac{2^{l}-1}{p}\right) \ll \left(\log X\right)^{2^{l}-1}.$$ 
\end{proof}
\begin{lemma}\label{trace} {\bf (Petersson trace formula)}
For any two natural numbers $m$ and $n,$ we have 
\begin{equation}\begin{split}
\sum_{f \in S_{k}} \frac{\lambda_{f}(n) \lambda_{f}(m)}{\|f\|^{2}}&=\frac{(4\pi)^{k-1}}{\Gamma(k-1)}\delta_{m=n}\\&+O\left(\frac{(4\pi)^{k-1}}{\Gamma(k-1)} (\log (3mn))^{2} \frac{\tau\big((m,n)\big)(mn)^{1/4}}{k^{1/2}} \right),   
\end{split}\end{equation}
where the implied constant is absolute, and $\|f\|$ denotes the Petersson norm of $f$ over $C_{k}$.
\end{lemma}
\begin{proof}
See \cite[Corollary 14.24, Theorem 16.7]{IK1}.
\end{proof}
\begin{lemma}\label{Kuznetsov}  {\bf (Kuznetsov trace formula)} 
Let $\{u_{j}\}$ be an orthonormal basis of Hecke-Maass cusp forms for $SL(2,\mathbb{Z}),$ and let $\frac{1}{4}+t_{j}^{2}$ be the Laplace eigenvalue of $u_{j}.$
Suppose $T>1,$ and let $m,n \in \mathbb{N}.$ 
Then, we have  
\begin{equation}
\sum_{j} w(j)e^{-t_{j}/T}\lambda_{u_{j}}(m)\lambda_{u_{j}}(n) =\frac{T^{2}}{6}\delta_{m=n} +O_{\epsilon}\left(T^{1+\epsilon}(mn)^{\epsilon} +(mn)^{1/2+\epsilon}\right)\end{equation}
where $w_{j}$ is defined in Theorem \ref{Theorem1.2}, and the summation over $j$ is taken over all $u_j$.
Note that 
$$\sum_{j} w(j)e^{-t_{j}/T} \sim \frac{T^{2}}{6}.$$
\end{lemma}
\begin{proof}
See \cite[Lemma 1]{BB1}. 
\end{proof}
\noindent 
Unlike holomorphic cusp forms, we do not have a suitable bound such as Deligne's bound for our purposes. Consequently, we cannot apply Lemma \ref{shiu}. Nevertheless, since we are dealing with families of Hecke-Maass forms, we can employ the Kuznetsov trace formula instead of Lemma \ref{shiu}.
\begin{lemma}\label{shiu2} For any $d_{0},\delta_{1},\delta_{2} \in \mathbb{N},$ we have
$$\sum_{j} w(j)e^{-t_{j}/T}\left|\lambda_{u_{j}}\left(\left(d_{0}\delta_{1}\right)^{2}\right)\lambda_{u_{j}}\left(\left(d_{0}\delta_{2}\right)^{2}\right)\right| \ll_{\epsilon} T^{2}\left(d_{0}^{2}\delta_{1}\delta_{2}\right)^{\epsilon}+\left(d_{0}^{2}\delta_{1}\delta_{2}\right)^{1/2+\epsilon}.$$
\end{lemma}
\begin{proof}
Using \eqref{heckemaass}, we obtain 
\begin{equation}\begin{split}\lambda_{u_{j}}\left(\left(d_{0}\delta_{i}\right)^{2}\right)&=\sum_{r|d_{0}\delta_{i}}\mu(r_{i})\lambda_{u_{j}}^{2}\left(d_{0}\delta_{i}/r_{i}\right)
\\&= \sum_{r|d_{0}\delta_{i}} \mu(r_{i})\sum_{k_{i} | \frac{d_{0}\delta_{i}}{r_{i}}}\lambda_{u_{j}}(k_{i}^{2})
\end{split}\nonumber\end{equation}
for $i=1,2.$ Note that the term $$\sum_{k_{i} | \frac{d_{0}\delta_{i}}{r_{i}}}\lambda_{u_{j}}(k_{i}^{2})$$ is positive. Therefore, by the triangle inequality, we see that 
$$\left|\lambda_{u_{j}}\left(\left(d_{0}\delta_{1}\right)^{2}\right)\lambda_{u_{j}}\left(\left(d_{0}\delta_{2}\right)^{2}\right)\right| 
\leq \sum_{r_{1}|d_{0}\delta_{1}} |\mu(r_{1})| \sum_{r_{2}|d_{0} \delta_{2}}|\mu(r_{2})|\sum_{k_{1}|\frac{d_{0}\delta_{1}}{r_{1}}}\sum_{k_{2}|\frac{d_{0}\delta_{2}}{r_{2}}}\lambda_{u_{j}}\left(k_{1}^{2}\right)\lambda_{u_{j}}\left(k_{2}^{2}\right). $$
By Lemma \ref{Kuznetsov}, we have
\begin{equation}\begin{split} &\sum_{j} w(j)e^{-t_{j}/T}\sum_{r_{1}|d_{0}\delta_{1}} |\mu(r_{1})| \sum_{r_{2}|d_{0} \delta_{2}}|\mu(r_{2})|\sum_{k_{1}|\frac{d_{0}\delta_{1}}{r_{1}}}\sum_{k_{2}|\frac{d_{0}\delta_{2}}{r_{2}}}\lambda_{u_{j}}\left(k_{1}^{2}\right)\lambda_{u_{j}}\left(k_{2}^{2}\right) 
\\&\ll_{\epsilon} T^{2} \sum_{r_{1}|d_{0}\delta_{1}} |\mu(r_{1})| \sum_{r_{2}|d_{0} \delta_{2}}|\mu(r_{2})|
\left(\mathop{\sum_{k_{1}|\frac{d_{0}\delta_{1}}{r_{1}}}\sum_{k_{2}|\frac{d_{0}\delta_{2}}{r_{2}}}}_{k_{1}=k_{2}} 1 + \sum_{k_{1}|\frac{d_{0}\delta_{1}}{r_{1}}}\sum_{\substack{k_{2}|\frac{d_{0}\delta_{2}}{r_{2}}}}\left( \frac{(k_{1}k_{2})^{\epsilon}}{T^{1-\epsilon}} +\frac{(k_{1}k_{2})^{1/2+\epsilon}}{T^{2}} \right)\right)
\\&\ll_{\epsilon} T^{2} \sum_{r_{1}|d_{0}\delta_{1}} |\mu(r_{1})| \sum_{r_{2}|d_{0} \delta_{2}}|\mu(r_{2})|\left( \tau\left(\frac{d_{0}\delta_{1}}{r_{1}}\right)\tau\left(\frac{d_{0}\delta_{2}}{r_{2}}\right)\left(1+\frac{\left(d_{0}^{2}\delta_{1}\delta_{2}\right)^{\epsilon}}{T^{1-\epsilon}}+\frac{\left(d_{0}^{2}\delta_{1}\delta_{2}\right)^{1/2+\epsilon}}{T^{2}}\right) \right).
\end{split}\end{equation}
Using the divisor bound $\tau(n) \ll_{\epsilon} n^{\epsilon},$
we can bound the above expression as follows:
$$\ll_{\epsilon} T^{2}\left(d_{0}^{2}\delta_{1}\delta_{2}\right)^{\epsilon}+\left(d_{0}^{2}\delta_{1}\delta_{2}\right)^{1/2+\epsilon}.$$
\end{proof}
\section{Propositions and the proof of Theorem 1.1}
\noindent Let $z$ represent $H^{1+\epsilon}$ from now on. We begin by splitting the integral in \eqref{thm} into three parts:
\begin{equation}\label{1} \int_{X}^{2X} \left(\sum_{\substack{x<n d \leq x+H \\ d \leq z}} \lambda_{f}(d^{2})-H \sum_{d \leq z} \frac{\lambda_{f}\left(d^{2}\right)}{d}\right)^2 d x,\end{equation} 
\begin{equation}\label{2} \int_{X}^{2X}\left(\sum_{\substack{x<n d \leq x+H \\ z < d}} \lambda_{f}(d^{2})\right)^2 d x,\end{equation} and 
\begin{equation}\label{3} \int_X^{2 X}\left(Hc_{1,f}-H\sum_{d \leq z} \frac{\lambda_{f}\left(d^{2}\right)}{d}\right)^2 d x.\end{equation}
Using the Cauchy-Schwarz inequality, we obtain an upper bound for the integral in \eqref{thm}:
$$ \frac{1}{X}\sum_{f \in S_{k}} \varphi_{f} \left(\eqref{1}+\eqref{2}+\eqref{3}\right).$$ Hence, we need to demonstrate that each of the three integrals is bounded by $k^{\epsilon}H^{1+\epsilon}$ on average over $f \in S_{k}.$
We analyze the contribution of equation \eqref{1} in Propositions 3.1 and 3.2 by utilizing the Petersson trace formula alongside the proof presented in \cite{GKMR}. Although it is possible to obtain Proposition 3.2 for individual $f$, the range of $H$ is restricted to $\log H < (\log X/4)$. In Proposition 3.3, we examine the contribution of equation \eqref{2} by applying the Petersson trace formula. In Proposition \ref{RankinProposition35}, we obtain a pointwise bound for
$$Hc_{1,f}-H\sum_{d \leq z} \frac{\lambda_{f}\left(d^{2}\right)}{d}$$ using the Lindel\"of-on-average bound to handle the contribution of equation \eqref{3}.
\subsection{Propositions}
The following proposition employs the same method as used in Proposition 5 of \cite{GKMR}. Like Proposition 5, we anticipate that the integral \eqref{1} will provide the dominant term.
\begin{proposition}\label{Proposition1} Let $1\ll \log H < \log X.$ As $X \rightarrow \infty$,
\begin{equation}\begin{split}
\sum_{f \in S_{k}} &\varphi_{f} \frac{1}{X} \int_{X}^{2X} \left(\sum_{\substack{x<n d \leq x+H \\ d \leq z}} \lambda_{f}(d^{2})-H \sum_{d \leq z} \frac{\lambda_{f}\left(d^{2}\right)}{d}\right)^2 d x 
\\&= \left(\frac{1}{\pi^{2}}+O\left(H^{-\epsilon/2}\right)\right)\sum_{d \leq z}  \sum_{0< n \leq X^{5}} \frac{1}{n^{2}} \sin^{2}\left(\frac{ \pi nH}{d}\right)+ O_{\epsilon}\left(z^{3+2\epsilon}(\log X)^{2} k^{-\frac{1}{2}}\right) + O\left((\log X)^{63}\right).  
\end{split}\nonumber\end{equation}
\end{proposition}
\begin{proof} Without loss of generality, we assume that $H \in \mathbb{N}.$ We have
\begin{equation}\begin{split}
\sum_{\substack{x \leq nd \leq x+H\\ d \leq z}} \lambda_{f}\left(d^{2}\right) 
&= \sum_{\substack{1 \leq d \leq x+H \\ d \leq z}}\lambda_{f}\left(d^{2}\right)\sum_{\frac{x}{d} < n \leq \frac{x+H}{d}}1 
\\&= \sum_{\substack{1 \leq d \leq x+H \\ d\leq z}} \lambda_{f}\left(d^{2}\right) \left(\frac{H}{d} + \psi\left(\frac{x}{d}\right)- \psi\left(\frac{x+H}{d}\right)\right)
\end{split}\end{equation}
where $\psi(x)=x-[x]-1/2.$
It is known that for any $N \in \mathbb{N}$ and $x \in \mathbb{R}$, we have 
$$\psi(x)= -\frac{1}{2\pi i} \sum_{0 < |n| < N} \frac{1}{n}e(nx) +O\left(\min\left(1,\frac{1}{N\|x\|}\right)\right)$$ 
(see \cite{IK1}[(4.18)]). We will choose the value of $N$ later.
Therefore, we see that 
\begin{equation}\label{lesszmain}\begin{split}
&\sum_{\substack{x<n d \leq x+H \\ d \leq z}} \lambda_{f}(d^{2})-H \sum_{d \leq z} \frac{\lambda_{f}\left(d^{2}\right)}{d}
\\&= -\frac{1}{2\pi i}\sum_{\substack{1 \leq d \leq x+H \\ d \leq z}} \lambda_{f}\left(d^{2}\right) \left(\sum_{0<|n|\leq N} \frac{1}{n}\left(e\left(\frac{nx}{d}\right)-e\left(\frac{n\left(x+H\right)}{d}\right)\right)\right)
    +O(\mathcal{M}(x,z)),
\end{split}\end{equation}
where 
\begin{equation}\begin{split}
  \mathcal{M}(x,z) &=\sum_{\substack{1 \leq d \leq x+H \\ d \leq z}} \left|\lambda_{f}\left(d^{2}\right)\right| \left(\min\left(1,\frac{1}{N\|\frac{x+H}{d}\|}\right)+\min\left(1,\frac{1}{N\|\frac{x}{d}\|}\right)\right).
\end{split}\end{equation}
Using the Cauchy-Schwarz inequality, we have 
\begin{equation}\label{error1}\begin{split}\frac{1}{X}&\sum_{X \leq x \leq 2X}\left(\sum_{1\leq d  \leq z} \left|\lambda_{f}\left(d^{2}\right)\right|\min\left(1,\frac{1}{N\|\frac{x+H}{d}\|}\right)\right)^{2}
\\&\ll\frac{1}{X}\sum_{X \leq x \leq 2X}\left(\sum_{\substack{d\mid x+H \\ d \leq z}}\left|\lambda_{f}\left(d^{2}\right)\right| + \sum_{\substack{d \nmid x+H \\ d \leq z}} \frac{d\left|\lambda_{f}\left(d^{2}\right)\right|}{N}\right)^{2}
\\&\ll\frac{1}{X}\sum_{X \leq x \leq 2X}\left(\left(\sum_{\substack{d\mid x+H \\ d \leq z}}\left|\lambda_{f}\left(d^{2}\right)\right|\right)^{2} + \left(\sum_{\substack{d \nmid x+H \\ d \leq z}} \frac{d\left|\lambda_{f}\left(d^{2}\right)\right|}{N}\right)^{2}\right).
\end{split}\end{equation}By Deligne's bound and using \eqref{average2}, we estimate the first term in the last inequality of \eqref{error1} as
\begin{equation}\begin{split}&\ll\frac{1}{X} \sum_{X \leq x \leq 2X} \tau(x+H)^{6}
\\& \ll(\log X)^{63}.\end{split}\end{equation}
Similarly, we estimate the second term as    
\begin{equation}\begin{split}&
\ll \frac{1}{X} \sum_{X \leq x \leq 2X} \left( \sum_{d \leq z} \frac{d\tau(d)^{2}}{N}\right)^{2}
\\&\ll \left(\frac{z^{2} \left(\log z\right)^{3}}{N}\right)^{2}.
\end{split}\end{equation}
Using the same approach as in \eqref{error1}, we have
\begin{equation}
\frac{1}{X}\sum_{X \leq x \leq 2X}\left(\sum_{1\leq d  \leq z} \left|\lambda_{f}\left(d^{2}\right)\right|\min\left(1,\frac{1}{N\|\frac{x}{d}\|}\right)\right)^{2} \ll (\log X)^{63} + \left(\frac{z^{2} \left(\log z\right)^{3}}{N}\right)^{2}.
\nonumber\end{equation}
In order to proceed, we choose $N=X^5.$
Then, we obtain
$$ \frac{1}{X}\int 1_{[1,2]}\left(\frac{x}{X}\right)\left|\mathcal{M}(x,z)\right|^{2} dx\ll (\log X)^{63}.$$
Next, we consider the contribution of the main term in \eqref{lesszmain}, which is given by
\begin{equation}\label{lesszaftermain}
    \frac{1}{4 \pi^{2} X} \int_{-\infty}^{\infty} 1_{[1,2]}\left(\frac{x}{X}\right) \left| \sum_{d \leq z} \lambda_{f}\left(d^{2}\right)\sum_{0<|n|<X^{5}}\frac{1}{n}e\left(\frac{nx}{d}\right)\left(1-e\left(\frac{nH}{d}\right)\right)\right|^{2}dx.
\end{equation}
Since $\sigma_{-} \leq 1_{[1,2]} \leq \sigma_{+},$ \eqref{lesszaftermain} lies between  
\begin{equation}\label{1lesszaftermain2}\begin{split}
       \frac{1}{4 \pi^{2}}  &\sum_{d_{1} \leq z}\sum_{d_{2} \leq z} \lambda_{f}\left(d_{1}^{2}\right) \lambda_{f}\left(d_{2}^{2}\right)  \sum_{0<|n_{1}|<X^{5}}\frac{1}{n_{1}}\sum_{0<|n_{2}|<X^{5}}\frac{1}{n_{2}}
       \\& \times \left(1-e\left(\frac{n_{1}H}{d_{1}}\right)\right)\overline{\left(1-e\left(\frac{n_{2}H}{d_{2}}\right)\right)}\hat{\sigma}_{-}\left(-X\left(\frac{n_{1}}{d_{1}}-\frac{n_{2}}{d_{2}}\right)\right)
\end{split}\end{equation} 
and 
\begin{equation}\label{lesszaftermain2}\begin{split}
       \frac{1}{4 \pi^{2}}  &\sum_{d_{1} \leq z}\sum_{d_{2} \leq z} \lambda_{f}\left(d_{1}^{2}\right) \lambda_{f}\left(d_{2}^{2}\right)  \sum_{0<|n_{1}|<X^{5}}\frac{1}{n_{1}}\sum_{0<|n_{2}|<X^{5}}\frac{1}{n_{2}}
       \\& \times \left(1-e\left(\frac{n_{1}H}{d_{1}}\right)\right)\overline{\left(1-e\left(\frac{n_{2}H}{d_{2}}\right)\right)}\hat{\sigma}_{+}\left(-X\left(\frac{n_{1}}{d_{1}}-\frac{n_{2}}{d_{2}}\right)\right).
\end{split}\end{equation} 
By applying the Petersson trace formula, we have that 
\begin{equation}\label{normalizedvarphi}\begin{split}
     &\sum_{f \in S_{k}} \varphi_{f} \sum_{d_{1} \leq z}\sum_{d_{2} \leq z} \lambda_{f}\left(d_{1}^{2}\right) \lambda_{f}\left(d_{2}^{2}\right)  \sum_{0<|n_{1}|<X^{5}}\frac{1}{n_{1}}\sum_{0<|n_{2}|<X^{5}}\frac{1}{n_{2}} \\& \;\;\;\;\;\; \times 
     \left(1-e\left(\frac{n_{1}H}{d_{1}}\right)\right)\overline{\left(1-e\left(\frac{n_{2}H}{d_{2}}\right)\right)}\hat{\sigma}_{\pm}\left(-X\left(\frac{n_{1}}{d_{1}}-\frac{n_{2}}{d_{2}}\right)\right)
     \\&= \sum_{d \leq z} 1 \sum_{0<|n_{1}|<X^{5}}\frac{1}{n_{1}}\sum_{0<|n_{2}|<X^{5}}\frac{1}{n_{2}}
     \left(1-e\left(\frac{n_{1}H}{d}\right)\right)\overline{\left(1-e\left(\frac{n_{2}H}{d}\right)\right)}\hat{\sigma}_{\pm}\left(-X\left(\frac{n_{1}-n_{2}}{d}\right)\right)
     \\&\;\;+ O_{\epsilon}\left(\sum_{d_{1} \leq z} \sum_{d_{2} \leq z} \left(d_{1}d_{2}\right)^{\frac{1}{2}+\epsilon}\sum_{0<|n_{1}|<X^{5}}\frac{1}{|n_{1}|}\sum_{0<|n_{2}|<X^{5}}\frac{1}{|n_{2}|} \left|\hat{\sigma}_{\pm}\left(-X\left(\frac{n_{1}}{d_{1}}-\frac{n_{2}}{d_{2}}\right)\right)\right|k^{-\frac{1}{2}}\right).
\end{split}\end{equation}

By applying the condition on the supports of $\hat{\sigma}_{-}$ and $\hat{\sigma}_{+}$, we see that 
\begin{equation}\begin{split}
   \hat{\sigma}_{\pm} \left(-X\left(\frac{n_{1}-n_{2}}{d}\right)\right) \neq 0 
   &\Rightarrow  \left|\frac{n_{1}-n_{2}}{d}\right| \leq \frac{B}{X}H^{\frac{\epsilon}{2}}
   \\&\Rightarrow \left|n_{1}-n_{2}\right|\leq \frac{BH^\frac{\epsilon}{2}d}{X}.
\end{split}\end{equation}
Assuming that $\epsilon$ is sufficiently small and $X$ is sufficiently large so that $d < z \ll_{\epsilon} X^{1-\epsilon}$, we can imply the following:
$$\hat{\sigma}_{\pm}\left(-X\left(\frac{n_{1}-n_{2}}{d}\right)\right) \neq 0 
\Rightarrow n_{1}=n_{2}.$$
Therefore, equation \eqref{normalizedvarphi} can be written as
\begin{equation}
   4\hat{\sigma}_{\pm}(0) \sum_{d \leq z} 1 \sum_{0<| n| \leq X^{5}} \frac{1}{n^{2}} \sin^{2}\left(\frac{ \pi nH}{d}\right)+ O_{\epsilon}\left(z^{3+2\epsilon} (\log X)^{2} k^{-\frac{1}{2}}\right).
\end{equation} Note that the main term is non-negative.
By applying the third property of $\sigma_{\pm}$ in Lemma \ref{smoothing}, the above term can be written as 
\begin{equation}
 4\left(1+O\left(H^{-\epsilon/2}\right)\right) \sum_{d \leq z} 1  \sum_{0< n \leq X^{5}} \frac{1}{n^{2}} \sin^{2}\left(\frac{ \pi nH}{d}\right)+ O_{\epsilon}\left(z^{3+2\epsilon}(\log X)^{2} k^{-\frac{1}{2}}\right)
\end{equation}
\end{proof}
\noindent The following proposition employs a similar method as in \cite[Lemma 9]{GKMR}.
\begin{proposition}\label{proposition2} Assume the condition in Proposition 3.1. As $X \rightarrow \infty$, we have 
\begin{equation}\label{sinsquare}\begin{split}
  \frac{1}{\pi^{2}}\sum_{d \leq z} 1  \sum_{0< n \leq X^{5}} \frac{1}{n^{2}} \sin^{2}\left(\frac{ \pi nH}{d}\right)
 \ll_{\epsilon} Hz^{\epsilon}
\end{split}\end{equation}  
\end{proposition}
\begin{proof}
Define the function $W(y)$ by 
$$
W(y)= \frac{\sin (\pi y)}{ \pi y} h\left(y / H^{\varepsilon / 4}\right),
$$
where $h$ is a smooth bump function such that $h(x)=1$ for $|x| \leq 1$ and $h(x)=0$ for $|x| \geq 2$. 
Then, we have 
\begin{equation}
\begin{aligned}
& H^2 \sum_{d \leq z} \frac{1}{d^{2}} \sum_{0<n \leq X^5}\left(\frac{\sin^{2}\left(\frac{ \pi n H}{d}\right)}{\left(\frac{ \pi nH}{d}\right)^{2}}-\left|W\left(\frac{nH}{d}\right)\right|^2\right) \\
&\ll  \sum_{d \leq z } 1  \sum_{0<n \leq X^{5}} \frac{ \delta_{\frac{nH}{d} \geq H^{\epsilon / 4}}}{n^{2}}
\\&\ll \sum_{d \leq H^{1-\epsilon/4}} 1 + \sum_{ H^{1-\epsilon/4} < d \leq z} \frac{H^{1-\epsilon/4}}{d}
\\&\ll H^{1-\epsilon/4}\log z.
\end{aligned}
\end{equation}
Let us consider $$H^{2} \sum_{d \leq z} \frac{1}{d^2}  \sum_{0 < n \leq X^5} \left|W\left(\frac{ nH}{d}\right)\right|^{2}.$$ 
Let $g(x)=\left|W(e^{x})\right|^{2}e^{x}.$ Then we have the following expression for the Fourier transform of $g(x)$:$$\hat{g}(x)=\int_{0}^{\infty}\left|W(y)\right|^{2}y^{-2\pi i x} dy,$$ where $\hat{g}(\xi)$ is entire and satisfies $\hat{g}(\xi)=O\left(H^{2 \varepsilon} /(|\xi|+1)^3\right)$ uniformly for $|\Im(\xi)|<1 /2 \pi$
(see \cite[(16)]{GKMR}). By the Mellin inversion, we get 
\begin{equation}\label{WR}|W(r)|^2=r^{-1} \frac{1}{2 \pi i} \int_{(c)} r^s \hat{g}\left(\frac{s}{2 \pi i}\right) d s\end{equation}
where $c \in (-1,1).$
By using \eqref{WR}, we have 
\begin{equation}\label{sinsquare2}\begin{split}
 &H^{2} \sum_{d\leq z} \frac{1}{d^2} \sum_{n \geq 1}  \left|W\left(\frac{nH}{d}\right)\right|^{2} 
 \\&= \frac{H}{2\pi i} \sum_{d \leq z}\frac{ 1}{d}  \sum_{n \geq 1} \frac{1}{n} \int_{(-2\epsilon)} \left(\frac{nH}{d}\right)^{s} \hat{g}\left(\frac{s}{2 \pi i}\right) ds 
 \\&= \frac{H}{2\pi i}  \sum_{d \leq z} \frac{1}{d}  \int_{(-2\epsilon)} \zeta(1-s) \left(\frac{H}{d}\right)^{s} \hat{g}\left(\frac{s}{2 \pi i}\right) ds.
  \end{split}\end{equation}  
By using the fact that
$|\zeta(1+2\epsilon+it)| \leq \zeta(1+2\epsilon) \ll_{\epsilon}1 \;\; \textrm{for any} \;\; t\in \mathbb{R},$ we see that 
 $$\int_{(-2\epsilon)} \zeta(1-s) \left(\frac{H}{d}\right)^{s} \hat{g}\left(\frac{s}{2 \pi i}\right) ds \ll_{\epsilon}d^{2\epsilon}.$$
Therefore, \eqref{sinsquare2} is bounded by 
 \begin{equation}\begin{split}&\ll_{\epsilon} H  \sum_{d \leq z}\frac{1}{d} d^{2\epsilon}
 \\&\ll_{\epsilon} Hz^{2\epsilon}\end{split}\end{equation}
\end{proof}
\begin{remark}
In contrast to Lemma 9 in \cite{GKMR}, we are unable to obtain an asymptotic estimate in this case because the exponent of $d$ in the second line of equation \eqref{sinsquare2} is $1$. However, we can anticipate that the order of $H$ is $1$.
\end{remark}
\begin{proposition}
 Assume the condition in Proposition 3.1. As $X \rightarrow \infty$,
\begin{equation}\label{biggerthanz}\begin{split}
   \sum_{f \in S_{k}} \varphi_{f} \frac{1}{X} &\int_X^{2 X}\left(\sum_{\substack{x<n d \leq x+H \\ d > z }} \lambda_{f}(d^{2})\right)^2 d x 
\ll_{\epsilon} \left(H \log \frac{2X+H}{z}+ H\frac{X^{2+\epsilon}}{k^{1/2}}\right).
\end{split}\end{equation}
\end{proposition}
\begin{proof}
Since $z>H,$ for each $x \in [X,2X],$ we have 
\begin{equation}\begin{split}\left(\sum_{\substack{x<n d \leq x+H \\ d > z }} \lambda_{f}(d^{2})\right)^2
=\left(\sum_{z< d \leq x+H}\lambda_{f}\left(d^{2}\right)  \delta_{[\frac{x}{d}]+1=[\frac{x+H}{d}]}\right)^{2}.\end{split}\end{equation}
By squaring out the sum, we express the above quantity as 
\begin{equation}
 \sum_{z< d_{1} \leq x+H}\sum_{z< d_{2} \leq x+H} \lambda_{f}\left(d_{1}^{2}\right)\lambda_{f}\left(d_{2}^{2}\right)
 \delta_{[\frac{x}{d_{1}}]+1=[\frac{x+H}{d_{1}}]} \delta_{[\frac{x}{d_{2}}]+1=[\frac{x+H}{d_{2}}]}.
\end{equation}
By Lemma \ref{trace}, we see that 
\begin{equation}\begin{split}
&\sum_{f \in S_{k}} \frac{1}{\|f\|^{2}}  \sum_{z< d_{1} \leq x+H}\sum_{z< d_{2} \leq x+H} \lambda_{f}\left(d_{1}^{2}\right)\lambda_{f}\left(d_{2}^{2}\right)
 \delta_{[\frac{x}{d_{1}}]+1=[\frac{x+H}{d_{1}}]} \delta_{[\frac{x}{d_{2}}]+1=[\frac{x+H}{d_{2}}]}
 \\& \ll_{\epsilon} \sum_{f \in S_{k}} \frac{1}{\|f\|^{2}}  \left(\sum_{z <d \leq x+H}  \delta_{[\frac{x}{d}]+1=[\frac{x+H}{d}]}+
\sum_{z< d_{1} \leq x+H}\sum_{z< d_{2} \leq x+H} \frac{\left(d_{1}d_{2}\right)^{1/2+\epsilon}}{k^{1/2}} \delta_{[\frac{x}{d_{1}}]+1=[\frac{x+H}{d_{1}}]} \delta_{[\frac{x}{d_{2}}]+1=[\frac{x+H}{d_{2}}]} 
 \right).
 \end{split}\end{equation}
For each interval $[a,a+d) \subset [X,2X+H],$ 
there are $H$ terms $b\in [a,a+d)$ such that 
\begin{equation}\delta_{[\frac{b}{d}]+1=[\frac{b+H}{d}]}=1,\end{equation} and $0$ for other terms. 
Therefore, by interchanging the order of summations, we have 
\begin{equation}\begin{split}
      &\sum_{f \in S_{k}} \frac{1}{X}\frac{1}{\|f\|^{2}} \int_{X}^{2X}\left(\sum_{\substack{x<n d \leq x+H \\ d> z }} \lambda_{f}(d^{2})\right)^2 d x 
      \\&\ll_{\epsilon} \frac{1}{X}\sum_{f \in S_{k}} \frac{1}{\|f\|^{2}}\left(   \sum_{z< d\leq 2X+H} H\frac{X}{d} + H+ H\frac{X^{3+2\epsilon}}{k^{1/2}}\right)
      \\&\ll_{\epsilon} \sum_{f \in S_{k}} \frac{1}{\|f\|^{2}} \left(H \log \frac{2X+H}{z}+ \frac{HX^{2+2\epsilon}}{k^{1/2}}\right).
\end{split}\end{equation}
\end{proof}
\begin{proposition}\label{RankinProposition35}  Assume the condition in Proposition 3.1 and the Lindel\"of-on-average bound \eqref{lindelof}. Then  
$$\sum_{f \in S_{k}} \varphi_{f} \frac{1}{X} \int_X^{2 X}\left|Hc_{1,f}-H\sum_{d \leq z} \frac{\lambda_{f}\left(d^{2}\right)}{d}\right|^2 d x= O_{\epsilon}\left(H^{2}k^{\epsilon}z^{-1+\epsilon}\right).$$ 
\end{proposition}
\begin{proof}
Let $T \in (1,X].$ Using Perron's formula (see \cite[Lemma 1.1]{Harman}) and Deligne's bound, we get 
\begin{equation}\label{abc}\sum_{d\leq z} \frac{\lambda_{f}\left(d^{2}\right)}{d}= \frac{1}{2\pi i} \int_{\frac{1}{\log z}-iT}^{\frac{1}{\log z}+iT} \frac{L\left(\textrm{sym}^{2}f,w+1\right)}{\zeta(2w+2)}\frac{z^w}{w}dw + O_{\epsilon}\left(\frac{z^{\epsilon}}{T} + z^{\epsilon-1}\right).  
\end{equation}
Applying H\"older's inequality, the Phragmen-Lindel\"of principle, and \eqref{lindelof},  we obtain 
\begin{equation}\begin{split} &\sum_{f \in S_{k}}  \varphi_{f} \left|\int_{\frac{1}{\log z}\pm iT}^{-\frac{1}{2}+\epsilon  \pm iT} \frac{L\left(\textrm{sym}^{2}f,w+1\right)}{\zeta(2w+2)}\frac{z^w}{w}dw\right|^{2} 
\\&\ll  \sum_{f \in S_{k}}  \varphi_{f} \int_{\frac{1}{\log z}\pm iT}^{-\frac{1}{2}+\epsilon  \pm iT}\left| \frac{L\left(\textrm{sym}^{2}f,w+1\right)}{\zeta(2w+2)}\frac{z^w}{w}\right|^2 dw
\\&\ll_{\epsilon} \frac{k^{\epsilon}}{T^{2-\epsilon}}.\end{split}\end{equation}
Therefore, by moving the line of the integral to the line $\textrm{Re}(w)=-1/2+\epsilon,$ we have
\begin{equation}\begin{split}  \sum_{f \in S_{k}}  \varphi_{f} \left( \sum_{d \leq z}\frac{\lambda_{f}\left(d^{2}\right)}{d}-c_{1,f} \right)^{2}
\ll_{\epsilon} &\sum_{f \in S_{k}} \varphi_{f} \left|\int_{-1/2+\epsilon-iT}^{-1/2+\epsilon+iT} \frac{L\left(\textrm{sym}^{2} f,w+1\right)}{\zeta(2w+2)}\frac{z^w}{w}dw\right|^{2} \\&\;\;\;+ \frac{z^{2\epsilon}}{T^2} + z^{2\epsilon-2}+\frac{k^{\epsilon}}{T^{2-\epsilon}}. \end{split}\end{equation}
By applying H\"older's inequality and \eqref{lindelof}, we obtain  
\begin{equation}\begin{split} \sum_{f \in S_{k}}&  \varphi_{f} \left|\int_{-1/2+\epsilon-iT}^{-1/2+\epsilon+iT} \frac{L\left(\textrm{sym}^{2} f,w+1\right)}{\zeta(2w+2)}\frac{z^w}{w}dw\right|^{2} 
\\&\ll \sum_{f \in S_{k}}  \varphi_{f} \int_{-1/2+\epsilon-iT}^{-1/2+\epsilon+iT}\left| \frac{L\left(\textrm{sym}^{2}f,w+1\right)}{\zeta(2w+2)}\frac{z^w}{w^{\frac{1}{2}-\epsilon}}\right|^{2}dw \int_{-1/2+\epsilon-iT}^{-1/2+\epsilon+iT}\left| \frac{1}{w^{\frac{1}{2}+\epsilon}}\right|^{2}dw 
\\&\ll_{\epsilon} z^{-1+2\epsilon}k^{\epsilon}T^{\epsilon}. \end{split}\end{equation}
By choosing $T=z^{1-\epsilon}$ and modifying $\epsilon,$ the proof is completed. 
\end{proof}
\subsection{Proof of Theorem 1.1}
By Proposition 3.1 and 3.2, the contribution from \eqref{1} is bounded by $$\ll_{\epsilon} \sum_{f \in S_{k}} \varphi_{f} \left(H^{1+\epsilon}z^{\epsilon}+ z^{3+2\epsilon}k^{-1/2} + (\log X)^{63}\right).$$ By Proposition 3.3  and the condition $\log X < (\log k)/4,$ the contribution from \eqref{2} is bounded by 
$$ \ll_{\epsilon} k^{\epsilon} H \log X.$$
By Proposition 3.4, the contribution from \eqref{3} is bounded by 
$$\ll_{\epsilon}  k^{\epsilon}H^{1+\epsilon}.$$
Since $\log X < (\log k)/4,$ these are bounded by 
$$ \ll_{\epsilon} k^{\epsilon}H^{1+\epsilon}.$$
\section{Propositions and the proof of Theorem 1.2}
The proof of Theorem 1.2 is similar to that of Theorem 1.1, but with some slight modifications. In particular, Propositions 4.1 and 4.2 use the Kuznetsov trace formula instead of the $SL(2,\mathbb{Z})$ Ramanujan-Petersson conjecture ($|\lambda_{u_{j}}(n)| \leq \tau(n)$). 

\subsection{Propositions}
\begin{proposition}\label{2Proposition1} Let $1\ll \log H < \log X,$ and let $2\log X - (\log H)/2 < \log T.$ As $X \rightarrow \infty$,
\begin{equation}\label{2propositionequation1}\begin{split}
\sum_{j} & \varpi_{j}e^{-t_{j}/T} \frac{1}{X} \int_{X}^{2X} \left(\sum_{\substack{x<n d \leq x+H \\ d \leq z}} \lambda_{u_{j}}(d^{2})-H \sum_{d \leq z} \frac{\lambda_{u_{j}}\left(d^{2}\right)}{d}\right)^2 d x 
\\&=\left(\frac{1}{\pi^2}+O\left(H^{-\epsilon / 2}\right)\right) \sum_{d \leq z} \sum_{0<n \leq X^5} \frac{1}{n^2} \sin ^2\left(\frac{\pi n H}{d}\right)
\\&+ \;\;\; O_\epsilon\left(z^{2+4 \epsilon}(\log X)^{2}T^{-1}+z^{4+4 \epsilon}(\log X)^{2}T^{-2}+X^{\epsilon}T^{\epsilon}\right) 
\end{split}\nonumber\end{equation}
\end{proposition}
\begin{proof} Without loss of generality, we assume that $H \in \mathbb{N}.$ Using the same argument as in the proof of Proposition \ref{Proposition1}, we have
\begin{equation}\label{2lesszmain}\begin{split}
&\sum_{\substack{x<n d \leq x+H \\ d \leq z}} \lambda_{u_{j}}(d^{2})-H \sum_{d \leq z} \frac{\lambda_{u_{j}}\left(d^{2}\right)}{d}
\\&= -\frac{1}{2\pi i}\sum_{\substack{1 \leq d \leq x+H \\ d \leq z}} \lambda_{u_{j}}\left(d^{2}\right) \left(\sum_{0<|n|\leq N} \frac{1}{n}\left(e\left(\frac{nx}{d}\right)-e\left(\frac{n\left(x+H\right)}{d}\right)\right)\right)
    +O(\mathcal{M'}(x,z)),
\end{split}\end{equation}
where 
\begin{equation}\begin{split}
  \mathcal{M'}(x,z) &=\sum_{\substack{1 \leq d \leq x+H \\ d \leq z}} \left|\lambda_{u_{j}}\left(d^{2}\right)\right| \left(\min\left(1,\frac{1}{N\|\frac{x+H}{d}\|}\right)+\min\left(1,\frac{1}{N\|\frac{x}{d}\|}\right)\right).
\end{split}\end{equation}
Using the Cauchy-Schwarz inequality, we see that
\begin{equation}\label{2error1}\begin{split}\frac{1}{X}&\sum_{X \leq x \leq 2X}\left(\sum_{1\leq d  \leq z} \left|\lambda_{u_{j}}\left(d^{2}\right)\right|\min\left(1,\frac{1}{N\|\frac{x+H}{d}\|}\right)\right)^{2}
\\&\ll\frac{1}{X}\sum_{X \leq x \leq 2X}\left(\sum_{\substack{d\mid x+H \\ d \leq z}}\left|\lambda_{u_{j}}\left(d^{2}\right)\right| + \sum_{\substack{d \nmid x+H \\ d \leq z}} \frac{d\left|\lambda_{u_{j}}\left(d^{2}\right)\right|}{N}\right)^{2}
\\&\ll\frac{1}{X}\sum_{X \leq x \leq 2X}\left(\left(\sum_{\substack{d\mid x+H \\ d \leq z}}\left|\lambda_{u_{j}}\left(d^{2}\right)\right|\right)^{2} + \left(\sum_{\substack{d \nmid x+H \\ d \leq z}} \frac{d\left|\lambda_{u_{j}}\left(d^{2}\right)\right|}{N}\right)^{2}\right).
\end{split}\end{equation}By squaring out  
the first term in the last inequality in \eqref{2error1} and applying Lemma \ref{shiu2} with the divisor bound $\tau(n)\ll_{\epsilon} n^{\epsilon},$ we get
\begin{equation}\label{skip}\begin{split}
   \sum_{j}  &w(j)e^{-t_{j}/T} \frac{1}{X}\sum_{X \leq x \leq 2X}\left(\sum_{\substack{d\mid x+H \\ d \leq z}}\left|\lambda_{u_{j}}\left(d^{2}\right)\right|\right)^{2}
   \\&\ll  \sum_{j}  w(j)e^{-t_{j}/T} \frac{1}{X}\sum_{X \leq x \leq 2X} \sum_{\substack{d_{1}\mid x+H \\ d_{1} \leq z}}\sum_{\substack{d_{2}\mid x+H \\ d_{2} \leq z}} \left|\lambda_{u_{j}}\left(d_{1}^{2}\right)\lambda_{u_{j}}\left(d_{2}^{2}\right)\right|
   \\&\ll_{\epsilon} T^{2}X^{\epsilon}z^{\epsilon}+X^{\epsilon}z^{1+\epsilon}.
\end{split}\end{equation}
Similarly, the contribution from the second term in the last inequality in \eqref{2error1} is   
\begin{equation}\begin{split}
   \sum_{j}  &w(j)e^{-t_{j}/T} \frac{1}{X}\sum_{X \leq x \leq 2X}\left(\sum_{\substack{d \nmid x+H \\ d \leq z}} \frac{d\left|\lambda_{u_{j}}\left(d^{2}\right)\right|}{N}\right)^{2}
   \\&\ll  \sum_{j}  w(j)e^{-t_{j}/T} \frac{1}{XN^{2}}\sum_{X \leq x \leq 2X} \sum_{ d_{1} \leq z}\sum_{d_{2} \leq z}  d_{1}d_{2}\left|\lambda_{u_{j}}\left(d_{1}^{2}\right)\lambda_{u_{j}}\left(d_{2}^{2}\right)\right|
   \\&\ll_{\epsilon} \frac{T^{2}X^{\epsilon}z^{4+\epsilon}+X^{\epsilon}z^{5+\epsilon}}{N^{2}}.
\end{split}\end{equation}
Similarly, we have 
\begin{equation}\begin{split}
 \sum_{j}  & w(j)e^{-t_{j}/T}\frac{1}{X}\sum_{X \leq x \leq 2X}\left(\sum_{1\leq d  \leq z} \left|\lambda_{u_{j}}\left(d^{2}\right)\right|\min\left(1,\frac{1}{N\|\frac{x}{d}\|}\right)\right)^{2} \\&\ll_{\epsilon} T^{2}X^{\epsilon}z^{\epsilon}+X^{\epsilon}z^{1+\epsilon} + \frac{T^{2}X^{\epsilon}z^{4+\epsilon}+X^{\epsilon}z^{5+\epsilon}}{N^{2}}.
\end{split}\nonumber\end{equation}
In order to proceed, we choose $N=X^{5}.$
By the condition  $ 2\log X -(\log H)/2 < \log T,$ we have 
$$  \sum_{j}  w(j)e^{-t_{j}/T}\frac{1}{X}\int 1_{[1,2]}\left(\frac{x}{X}\right)\left|\mathcal{M'}(x,z)\right|^{2} dx\ll_{\epsilon} T^{2}X^{\epsilon}.$$
Using the same argument as in the proof of Proposition \ref{Proposition1}, we can express the contribution of the main term in \eqref{2lesszmain} as
\begin{equation}
  \left(\frac{1}{\pi^{2}}+O\left(H^{-\epsilon / 2}\right)\right) \sum_{d \leq z} 1 \sum_{0<n \leq X^5} \frac{1}{n^2} \sin ^2\left(\frac{\pi n H}{d}\right)+O_\epsilon\left(z^{2+4 \epsilon}(\log X)^{2} T^{-1} + z^{4+4\epsilon}(\log X)^{2}T^{-2}\right).
\end{equation}
\end{proof}
\begin{proposition}
 Assume the conditions in Proposition \ref{2Proposition1}. As $X \rightarrow \infty$,
\begin{equation}\label{2biggerthanz}\begin{split}
  \sum_{j} & \varpi_{j}e^{-t_{j}/T} \frac{1}{X} \int_X^{2 X}\left(\sum_{\substack{x<n d \leq x+H \\ d > z }} \lambda_{u_{j}}(d^{2})\right)^2 d x 
\\& \ll_{\epsilon}T^{\epsilon}H^{1+\epsilon}.
\end{split}\end{equation}
\end{proposition}
\begin{proof}
Since $z>H,$ for each $x \in [X,2X],$ we have 
\begin{equation}\begin{split}\left(\sum_{\substack{x<n d \leq x+H \\ d > z }} \lambda_{u_{j}}(d^{2})\right)^2
=\left(\sum_{z< d \leq x+H}\lambda_{u_{j}}\left(d^{2}\right)  \delta_{[\frac{x}{d}]+1=[\frac{x+H}{d}]}\right)^{2}.\end{split}\end{equation}
By squaring out the sum, we express the above quantity as 
\begin{equation}
 \sum_{z< d_{1} \leq x+H}\sum_{z< d_{2} \leq x+H} \lambda_{u_{j}}\left(d_{1}^{2}\right)\lambda_{u_{j}}\left(d_{2}^{2}\right)
 \delta_{[\frac{x}{d_{1}}]+1=[\frac{x+H}{d_{1}}]} \delta_{[\frac{x}{d_{2}}]+1=[\frac{x+H}{d_{2}}]}.
\end{equation}
By Lemma \ref{Kuznetsov}, we observe that 
\begin{equation}\begin{split}
& \sum_{j}  \varpi_{j}e^{-t_{j}/T}  \sum_{z< d_{1} \leq x+H}\sum_{z< d_{2} \leq x+H} \lambda_{u_{j}}\left(d_{1}^{2}\right)\lambda_{u_{j}}\left(d_{2}^{2}\right)
 \delta_{[\frac{x}{d_{1}}]+1=[\frac{x+H}{d_{1}}]} \delta_{[\frac{x}{d_{2}}]+1=[\frac{x+H}{d_{2}}]}
 \\& \ll_{\epsilon} \left(\sum_{z <d \leq x+H}  \delta_{[\frac{x}{d}]+1=[\frac{x+H}{d}]}+
\sum_{z< d_{1} \leq x+H}\sum_{z< d_{2} \leq x+H} \left(\frac{\left(d_{1}d_{2}\right)^{\epsilon}}{T^{1-\epsilon}} +\frac{\left(d_{1}d_{2}\right)^{1/2+\epsilon}}{T^{2}}\right) \delta_{[\frac{x}{d}]+1=[\frac{x+H}{d}]}
 \right).
 \end{split}\end{equation}
For each interval $[a,a+d) \subset [X,2X+H],$ 
there are $H$ terms $b\in [a,a+d)$ such that 
\begin{equation}\delta_{[\frac{b}{d}]+1=[\frac{b+H}{d}]}=1,\end{equation} and $0$ for other terms. 
Therefore, by interchanging the order of summations, we have  
\begin{equation}\begin{split}
      & \sum_{j}  \varpi_{j}e^{-t_{j}/T}\frac{1}{X} \int_{X}^{2X}\left(\sum_{\substack{x<n d \leq x+H \\ d> z }} \lambda_{u_{j}}(d^{2})\right)^2 d x 
      \\&\ll_{\epsilon} T^{\epsilon}\frac{1}{X}\left(   \sum_{z< d\leq 2X+H} H\frac{X}{d} + H+ H\frac{X^{2+\epsilon}}{T^{1-\epsilon}}+ H\frac{X^{3+\epsilon}}{T^{2}} \right)
      \\&\ll_{\epsilon}T^{\epsilon} \left(H \log \frac{2X+H}{z}+ H\frac{X^{1+\epsilon}}{T^{1-\epsilon}}+H\frac{X^{2+\epsilon}}{T^{2}}\right).
\end{split}\end{equation}
Since $2\log X - (\log H)/2 < \log T,$ the proof is completed.
\end{proof}
\begin{proposition}  Assume the conditions in Proposition \ref{2Proposition1}. Then  
$$ \sum_{j}  \varpi_{j}e^{-t_{j}/T}\frac{1}{X} \int_X^{2 X}\left|Hc_{1,u_{j}}-H\sum_{d \leq z} \frac{\lambda_{u_{j}}\left(d^{2}\right)}{d}\right|^2 d x= O_{\epsilon}\left(T^{\epsilon}H^{2}z^{-1+\epsilon}\right).$$ 
\end{proposition}
\begin{proof}
Let $U \in (1,X].$ Using Perron's formula (see \cite[Lemma 1.1]{Harman}) and the Kim-Sarnak bound $|\lambda_{u_{j}}(n)|\ll n^{7/64}$ (\cite{Henry}), we see that 
\begin{equation}\sum_{d\leq z} \frac{\lambda_{u_{j}}\left(d^{2}\right)}{d}= \frac{1}{2\pi i} \int_{\frac{1}{\log z}-iU}^{\frac{1}{\log z}+iU} \frac{L\left(\textrm{sym}^{2}u_{j},w+1\right)}{\zeta(2w+2)}\frac{z^w}{w}dw + O_{\epsilon}\left(\frac{z^{7/32+\epsilon}}{U} + z^{7/32+\epsilon-1}\right),  
\end{equation}where $w(j)\ll_{\epsilon} |t_{j}+1|^{\epsilon}$ (see \cite{Hoffstein}).
Using H\"older's inequality, the Phragmen-Lindel\"of principle, and \eqref{khanyoung},  we have 
\begin{equation}\begin{split} &\sum_{j}  w(j)e^{-t_{j}/T}\left|\int_{\frac{1}{\log z}\pm iU}^{-\frac{1}{2}+\epsilon  \pm iU} \frac{L\left(\textrm{sym}^{2}u_{j},w+1\right)}{\zeta(2w+2)}\frac{z^w}{w}dw\right|^{2} 
\\&\ll  \sum_{j}  w(j)e^{-t_{j}/T} \int_{\frac{1}{\log z}\pm iU}^{-\frac{1}{2}+\epsilon  \pm iU}\left| \frac{L\left(\textrm{sym}^{2}u_{j},w+1\right)}{\zeta(2w+2)}\frac{z^w}{w}\right|^2 dw
\\&\ll_{\epsilon} \frac{T^{2+\epsilon}}{U^{2}}.\end{split}\end{equation}
Therefore, by shifting the line of the integral to the line $\textrm{Re}(w)=-1/2+\epsilon,$ we have
\begin{equation}\begin{split}  \sum_{j}&  w(j)e^{-t_{j}/T} \left( \sum_{d \leq z}\frac{\lambda_{u_{j}}\left(d^{2}\right)}{d}-c_{1,u_{j}} \right)^{2}
\\&\ll_{\epsilon} \sum_{j}  w(j)e^{-t_{j}/T}\left|\int_{-1/2+\epsilon-iU}^{-1/2+\epsilon+iU} \frac{L\left(\textrm{sym}^{2}u_{j},w+1\right)}{\zeta(2w+2)}\frac{z^w}{w}dw\right|^{2} \\&\;\;\;+ T^{2}\left(\frac{z^{7/16+\epsilon}}{U^2} + z^{7/16+\epsilon-2}+\frac{T^{\epsilon}}{U^2}\right). \end{split}\end{equation}
Using H\"older's inequality and \eqref{khanyoung}, we see that  
\begin{equation}\begin{split} \sum_{j}  & w(j)e^{-t_{j}/T} \left|\int_{-1/2+\epsilon-iU}^{-1/2+\epsilon+iU} \frac{L\left(\textrm{sym}^{2}u_{j},w+1\right)}{\zeta(2w+2)}\frac{z^w}{w}dw\right|^{2} 
\\&\ll \sum_{j}  w(j)e^{-t_{j}/T} \int_{-1/2+\epsilon-iU}^{-1/2+\epsilon+iU}\left| \frac{L\left(\textrm{sym}^{2}u_{j},w+1\right)}{\zeta(2w+2)}\frac{z^w}{w^{\frac{1}{2}-\epsilon}}\right|^{2}dw \int_{-1/2+\epsilon-iU}^{-1/2+\epsilon+iU}\left| \frac{1}{w^{\frac{1}{2}+\epsilon}}\right|^{2}dw 
\\&\ll_{\epsilon} z^{-1+\epsilon}U^{\epsilon}T^{2+\epsilon}. \end{split}\end{equation}
By choosing $U=z^{1+\epsilon}$ and modifying $\epsilon,$ the proof is completed.
\end{proof}
\subsection{Proof of Theorem 1.2}
By combining the bounds from Proposition 3.2, 4.1, 4.2, and 4.3, we get
$$\sum_{T \leq t_{j} \leq 2T} \varpi_{j}e^{-t_{j}/T}
\frac{1}{X} \int_X^{2 X}\left(\sum_{x<m \leq x+H} \lambda_{u_{j}}^2(m)-c_{1,u_{j}}H\right)^2 d x\ll_{\epsilon} T^{\epsilon}H^{1+\epsilon}.$$
Since $e^{-t_{j}/T} \asymp 1$ for $t_{j} \in [T,2T],$ the proof is completed.

\subsection*{Acknowledgements} The author would like to express gratitude to Professor Xiaoqing Li for her unwavering support. Additionally, we extend our appreciation to the anonymous referee for providing us with a multitude of valuable suggestions, which have significantly enhanced the quality of our results. The author is supported by the Doctoral Dissertation Fellowship from the Department of Mathematics at the
State University of New York at Buffalo. 
\section{Data Availability} Data sharing not applicable to this article as no data sets were generated or analysed during the current study.
\section{Declarations}
\noindent No funding was received for conducting this study. The authors have no relevant financial or non-financial interests to disclose. 

\section{Appendix}
In this section, we only consider the variance of Hecke eigenvalues of holomorphic cusp forms over arithmetic progressions modulo prime numbers. However, it is straightforward to extend these results to Hecke-Maass cusp forms and to any arithmetic progression. 
Assuming $q$ is a prime number and $(a,q)=1$, we can use the orthogonality of Dirichlet characters to rewrite the sum
$$\sum_{n=1 \atop n \equiv a \;\textrm{mod} \;q}^{X} \lambda_{f}(n)^{2}$$ as
$$ \frac{1}{\phi(q)}\sum_{\chi}
\bar{\chi}(a)\sum_{n=1}^{X} \lambda_{f}(n)^{2}\chi(n),$$
where the sum over $\chi$ runs over all Dirichlet characters modulo $q.$ We can treat the sum over nontrivial characters as an error term.
Therefore, we have
$$\sum_{n=1 \atop n \equiv a \;\textrm{mod} \;q}^{X} \lambda_{f}(n)^{2}-\frac{1}{\phi(q)} \sum_{n=1 \atop (n,q)=1}^{X} \lambda_{f}(n)^{2}= \frac{1}{\phi(q)} \sum_{\chi \neq \chi_{0}} \chi(a) \sum_{n=1}^{X} \lambda_{f}(n)^{2}\chi(n).$$
By modifying the argument in \cite[Lemma 3.4, Remark 3.5]{Kim2023}, we can estimate the sum over $n$ coprime to $q$ as
\begin{equation}\label{maint}\frac{1}{\phi(q)} \sum_{n=1 \atop (n,q)=1}^{X} \lambda_{f}(n)^{2}= \frac{w_{f,q}}{\phi(q)}X + E_{f,q},\end{equation}
where 
$$E_{f,q} = O_{\epsilon}\left( \frac{X^{1/2}\log X (\log q)^{7}}{\phi(q)} \left(\int_{(1/2)} \bigg|\sum_{\substack{n=1 \\\left(n, q\right)=1}}^{\infty} \frac{\lambda_f\left(n\right)^2}{n^s} \frac{1}{s}\bigg|^{2}ds\right)^{1/2} +\frac{(\log q)^{7}}{\phi(q)}X^{1/2+\epsilon} \right),$$
and for some constant $w_{f,q}$.
Therefore, we have 
$$\left|\sum_{n=1 \atop n \equiv a \;\textrm{mod} \;q}^{X} \lambda_{f}(n)^{2}-\frac{w_{f,q}}{\phi(q)}X\right|^{2}=\left|E_{f,q}+\frac{1}{\phi(q)}\sum_{\chi \neq \chi_{0}}
\bar{\chi}(a)\sum_{n=1}^{X} \lambda_{f}(n)^{2}\chi(n)\right|^{2}.$$
Expanding the square, we obtain
\begin{equation}\begin{split}&\left|E_{f,q}+\frac{1}{\phi(q)}\sum_{\chi \neq \chi_{0}}
\bar{\chi}(a)\sum_{n=1}^{X} \lambda_{f}(n)^{2}\chi(n)\right|^{2}
\\&= |E_{f,q}|^{2}+ 2Re\left(E_{f,q}\frac{1}{\phi(q)}\sum_{\chi \neq \chi_{0}}
\bar{\chi}(a)\sum_{n=1}^{X} \lambda_{f}(n)^{2}\chi(n)\right) + \left|\frac{1}{\phi(q)}\sum_{\chi \neq \chi_{0}}
\bar{\chi}(a)\sum_{n=1}^{X} \lambda_{f}(n)^{2}\chi(n)\right|^{2}.\end{split}\end{equation}
Therefore, we see that 
\begin{equation}\begin{split}
    &\frac{1}{q-1}\sum_{a=1}^{q-1}\left|\sum_{n=1 \atop n \equiv a \;\textrm{mod} \;q}^{X} \lambda_{f}(n)^{2}-\frac{w_{f,q}}{\phi(q)}X\right|^{2}  
\\&= E_{f,q}^{2} + \frac{1}{q-1}\sum_{a=1}^{q-1}  \left|\frac{1}{\phi(q)}\sum_{\chi \neq \chi_{0}}
\bar{\chi}(a)\sum_{n=1}^{X} \lambda_{f}(n)^{2}\chi(n)\right|^{2}.
\end{split}\end{equation} 
We define $G(n,m)$ as follows: if $q$ divides $n-m$, then $G(n,m) = q-2$; otherwise, $G(n,m) = -1$.
By expanding and squaring out the last term, we obtain
$$\frac{1}{q-1}\sum_{a=1}^{q-1}  \left|\frac{1}{\phi(q)}\sum_{\chi \neq \chi_{0}}
\bar{\chi}(a)\sum_{n=1 \atop (n,q)=1}^{X} \lambda_{f}(n)^{2}\chi(n)\right|^{2} = \frac{1}{\phi(q)^{2}} \sum_{n=1 \atop (n,q)=1}^{X} \sum_{m=1 \atop (m,q)=1}^{X} \lambda_{f}(n)^{2}\lambda_{f}(m)^{2}G(n,m).$$
Using the Petersson trace formula, the average of the right-hand side of the above inequality is given by
\begin{equation}\begin{split}
    \frac{1}{\phi(q)^2}\sum_{d=1 \atop (d,q)=1}^{X} \sum_{n=1 \atop (n,q)=1}^{\frac{X}{d}}\sum_{m=1 \atop (m,q)=1}^{\frac{X}{d}}G(dn,dm)&=\frac{1}{\phi(q)^2}\sum_{d=1 \atop (d,q)=1}^{X} \sum_{n=1 \atop (n,q)=1}^{\frac{X}{d}}\sum_{m=1 \atop (m,q)=1}^{\frac{X}{d}}G(n,m)
    \\&\ll \frac{1}{\phi(q)^{2}} \sum_{d=1 \atop (d,q)=1}^{X} \sum_{n=1 \atop (n,q)=1}^{\frac{X}{d}}q
    \\& \ll \frac{1}{q} \left( X (\log X)\right). 
\end{split}\end{equation}
By \eqref{lindelof}, we see that
$$   \sum_{f \in S_{k}} \varphi_{f} \int_{(1/2)} \bigg|\sum_{\substack{n=1 \\\left(n, q_1\right)=1}}^{\infty} \frac{\lambda_f\left(q_0 n\right)^2}{n^s} \frac{1}{s}\bigg|^{2}ds \ll_{\epsilon} \sum_{f \in S_{k}} \varphi_{f} k^{\epsilon}. $$
Therefore, 
\begin{equation}\begin{split}
    \sum_{f \in S_{k}} \varphi_{f} |E_{f,q}|^{2} \ll_{\epsilon} k^{\epsilon} X^{1+\epsilon}q^{-2+\epsilon}. 
\end{split}\end{equation} Thus, the average variance of Hecke eigenvalues of holomorphic cusp forms over arithmetic progressions is bounded by $k^{\epsilon}X^{1+\epsilon}q^{-1+\epsilon}$ for families $f \in S_{k}$, provided that $q$ is less than $X^{1-\epsilon}$.

\bibliographystyle{siam}   
\bibliography{over}  
\end{document}